\documentclass[11pt]{amsart} 
\usepackage[utf8]{inputenc}
\usepackage{fullpage}

\usepackage[english]{babel}
\usepackage[left=1 in, right=1 in, top=1.4 in, bottom=1.3 in]{geometry}
\usepackage{amsmath}
\usepackage{amssymb}
\usepackage{mathtools}
\usepackage[utf8]{inputenc}
\usepackage{color}
\usepackage{amsthm}
\usepackage{tikz}
\usetikzlibrary{graphs,arrows.meta}
\usepackage[shortlabels]{enumitem}
\usepackage{graphicx}
\usepackage{caption}
\usepackage{subcaption}
\usepackage{hyperref}
\usepackage{tikz}
\usepackage{ytableau}
\usepackage{mathtools}
\usepackage[normalem]{ulem}

\newcommand{\vanish}[1]{}

\theoremstyle{definition}
\newtheorem{definition}{Definition}
\newtheorem{proposition}[definition]{Proposition}
\newtheorem{theorem}[definition]{Theorem}
\newtheorem{corollary}[definition]{Corollary}
\newtheorem{example}[definition]{Example} 

\newcommand{\bsy}{\boldsymbol}

 \usepackage[colorinlistoftodos]{todonotes}

\title{On the Limiting Vacillating Tableaux for Integer Sequences}

\author{Zhanar Berikkyzy}
\author{Pamela E.~Harris}
\author{Anna Pun}  
\author{Catherine Yan}
\author{Chenchen Zhao} 

\address[Berikkyzy]{Department of Mathematics, Fairfield University, Fairfield, CT 06824 } 
\email{zberikkyzy@fairfield.edu}

\address[Harris]{Department of Mathematical Sciences, University of Wisconsin Milwaukee, Milwaukee, WI 53211} 
\email{peharris@uwm.edu}

\address[Pun]{Department of Mathematics, CUNY Baruch College, New York, NY 10010} 
\email{anna.pun@baruch.cuny.edu} 

\address[Yan]{Department of Mathematics, Texas A\&M University, College Station, TX 77843}
\email{huafei-yan@tamu.edu}  

\address[Zhao]{Department of Mathematics, University of Southern California, Los Angeles, CA 90089} 
\email{zhao109@usc.edu}

\date{\today}

\begin{document}
\maketitle 

\begin{abstract}
A fundamental identity in the representation theory of the partition algebra is 
$n^k = \sum_{\lambda} f^\lambda m_k^\lambda$ for $n \geq 2k$, 
where $\lambda$ ranges over integer partitions of $n$, $f^\lambda$ is the number of standard Young tableaux of shape $\lambda$, and $m_k^\lambda$ is the number of vacillating tableaux of shape $\lambda$ and length $2k$. 
Using a combination of RSK insertion and jeu de taquin,  Halverson and Lewandowski constructed a bijection $DI_n^k$ that maps each integer sequence in $[n]^k$ to a pair consisting of a standard Young tableau and a vacillating tableau. In this paper, we show that for a given integer sequence $\bsy{i}$, when $n$ is sufficiently large, the vacillating tableaux determined by $DI_n^k(\bsy{i})$ become stable when $n \rightarrow \infty$; the limit is called the limiting vacillating tableau 
for $\bsy{i}$.  We give a characterization of the set of limiting vacillating tableaux   and present explicit formulas that enumerate those  vacillating tableaux. 


\end{abstract}

\section{Introduction}

In \cite{HL05} Halverson and Lewandowski presented combinatorial proofs  of two identities arising in the representation theory of the partition algebra $\mathcal{C}A_k(n)$.   One of the  identities is 
that for $n \geq 2k$, 
\begin{equation} \label{eq: Identity1} 
    n^k=\sum_{\lambda\vdash n} f^\lambda m_k^\lambda, 
\end{equation}
where the sum is over integer partitions $\lambda$ of $n$, $f^\lambda$ is the number of standard Young tableaux (SYT) of shape $\lambda$, and $m_k^\lambda$ is the number of vacillating tableaux of shape $\lambda$ and length $2k$, whose definition is given below.  To emphasize the dependency on the parameter $n$, 
we call such tableaux $n$-vacillating tableaux. 

\begin{definition} \cite{HL05} 
For $n \geq 2k$, an \emph{$n$-vacillating tableau of shape $\lambda$ and length $2k$} is a sequence of $2k+1$ integer partitions 
\[
((n)=\lambda^{(0)}, \lambda^{(\frac{1}{2})}, \lambda^{(1)}, \lambda^{(1\frac{1}{2})}, \dots, \lambda^{(k-\frac{1}{2})}, \lambda^{(k)}=\lambda), 
\]
satisfying the conditions that for each integer $j$,
\begin{enumerate}[(a)]
    \item $\lambda^{(j)} \supseteq  \lambda^{(j+\frac12)}$ and 
    $| \lambda^{(j)} / \lambda^{(j+\frac12)}|=1$, 
    \item $\lambda^{(j+\frac12)}\subseteq \lambda^{(j+1)}$ and 
    $|\lambda^{(j+1)}/ \lambda^{(j+\frac12)}|=1$. 
    \end{enumerate} 
    \end{definition} 
Note that the above conditions imply that $\lambda^{(j)} \vdash n$ and $\lambda^{(j+\frac12)} \vdash (n-1)$, 
where for a partition $\lambda=(\lambda_1, \dots, \lambda_t)$, $\lambda \vdash n$ means that $\lambda$ is an integer partition of $n$; i.e., 
$|\lambda|:=\lambda_1+ \cdots +\lambda_t =n$. 

 \begin{example} \label{Example-VT1} For $n = 7$ and $k=3$, the following sequence represents an $n$-vacillating tableau of shape $\lambda = (5,2)$ and length $2k$: \[\left(\ytableausetup{boxsize=7px}\ydiagram{7} , \ \ydiagram{6} , \ \ydiagram{6,1} , \  \ydiagram{5,1} , \  \ydiagram{5,2} , \ \ydiagram{4,2} , \  \ydiagram{5,2}\right).\]
\end{example} 
 
Let $\mathcal{VT}_{n,k}(\lambda)$ denote the set of $n$-vacillating tableaux of shape $\lambda$ and  length $2k$. Note that $\mathcal{VT}_{n,k}(\lambda)$ is empty if $\lambda$ has more than $k$ boxes below the first row. 
The bijective proof of \eqref{eq: Identity1}  given 
in \cite{HL05} 
 uses a combination of the Robinson-Schensted-Knuth (RSK) 
 insertion algorithm and jeu de taquin. 
For any integer sequence $\bsy{i}=(i_1, i_2, \dots, i_k) \in [n]^k$ where $[n]=\{1, 2, \dots, n\}$, Halverson and Lewandowski define an iterative ``delete-insert" process that  associates to $\bsy{i}$ a pair $DI_n^k(\bsy{i})=(T^k_{n}(\bsy{i}), P^k_{n}(\bsy{i}))$, where  $T^k_{n}(\bsy{i})$ is a SYT  of some shape $\lambda \vdash n$ and $P^k_{n}(\bsy{i})\in \mathcal{VT}_{n,k}(\lambda)$.

The paper \cite{HL05} contains an equivalent notion of vacillating tableaux that does not use the parameter $n$, which was  originally introduced in \cite{CDDSY07} to study the crossings and nestings of matchings and set partitions. To distinguish from the $n$-vacillating tableaux, we call this second notion the \emph{simplified vacillating tableaux}. 
Given an integer partition $\lambda=(\lambda_1, \lambda_2, \dots, \lambda_t) \vdash n$, 
let $\lambda^*$ be obtained from $\lambda$ by removing the first part $\lambda_1$, i.e., 
$\lambda^*=(\lambda_2, \dots, \lambda_t) \vdash (n-\lambda_1)$.
For an $n$-vacillating tableau $P_\lambda= ( \lambda^{(j)}: j=0, \frac{1}{2}, 1, 1\frac{1}{2}, \dots, k)$ in $\mathcal{VT}_{n,k}(\lambda)$, the \emph{simplified vacillating tableau $(P_{\lambda})^*$} is the sequence 
$( \mu^{(j)}: \mu^{(j)}= (\lambda^{(j)})^* \text{ for  } j=0, \frac{1}{2}, 1, 1\frac{1}{2}, \dots, k)$. 
One can also define the simplified vacillating tableaux  directly in terms of integer partitions. 

\begin{definition} \cite{CDDSY07}
A \emph{simplified vacillating tableau}  $P_k^*(\mu)$ of shape $\mu$  and length $2k$  is a sequence $(\mu^{(j)}: j = 0, \frac{1}{2}, 1, 1\frac{1}{2}, \dots, k)$ of integer partitions such that $\mu^{(0)} = \emptyset$, $\mu^{(k)} = \mu$, and for each integer $j = 0, 1, \dots, k-1$:
\begin{enumerate}[(a)]
\item $\mu^{(j)} \supseteq  \mu^{(j+\frac12)}$ and 
    $| \mu^{(j)} / \mu^{(j+\frac12)}|=0$ or $1$, 
    \item $\mu^{(j+\frac12)}\subseteq \mu^{(j+1)}$ and 
    $|\mu^{(j+1)}/ \mu^{(j+\frac12)}|=0$ or $1$. 
\end{enumerate}
 \end{definition} 

Note that the definition forces 
$\mu^{(\frac{1}{2})}=\emptyset$. 
The following is the simplified vacillating tableau corresponding to the sequence in Example~\ref{Example-VT1}: \[\left(\ytableausetup{boxsize=7px}\emptyset , \  \emptyset , \ \ydiagram{1}, \  \ydiagram{1} , \  \ydiagram{2}, \  \ydiagram{2}, \ \ydiagram{2}\right).\] 

 For an integer partition $\mu$ of at most $k$ boxes, 
denote by $\mathcal{SVT}_k(\mu)$ the set of simplified vacillating tableaux of shape $\mu$ and length $2k$. 
If $n \geq 2k$, then there is a unique $\lambda \vdash n$ such that $\mu = \lambda^*$ and 
 $P_\lambda \leftrightarrow (P_\lambda)^*$ is a bijection between $\mathcal{VT}_{n,k}(\lambda)$ and $\mathcal{SVT}_k(\mu)$.  Hence the two notions of vacillating tableaux are equivalent.

Vacillating tableaux  play an important role in both
representation theory and combinatorics.  The notion of $n$-vacillating tableaux was introduced as part of the image of the 
 delete-insert process $DI_n^k$, which gave a combinatorial analogue of the Schur-Weyl duality between the symmetric group algebra and the partition algebra. Building upon this, 
Benkart, Halverson and Harman \cite{BHH17} generalized the result  to various analogs of the partition algebra  and obtained explicit formulas for the dimensions  of the irreducible modules for symmetric group centralizer algebras. 
Recently, 
the Schur-Weyl-like duality  has led to exciting developments in the representation theory of partition algebras and the invariant theory of the symmetric group. Several papers, including \cite{BH17, COSSZ,HJ18},  
have explored these new results in detail. 

From a combinatorial perspective, vacillating tableaux can be viewed as specific walks on Young's lattice, which is the poset of all integer partitions ordered by containment of diagrams. A closely related concept is the \emph{up-down tableaux} or \emph{oscillating tableaux}, which are sequences of integer partitions starting with the empty shape and in which consecutive partitions differ by one square. Up-down tableaux were originally introduced by Berele \cite{Berele86} in 1986 to describe the irreducible representations of the symplectic group $Sp(2k)$ and can be identified as a subfamily of vacillating tableaux. Stanley established a bijection between complete matchings on $[2k]$ and up-down tableaux of empty shape and length $2k$, which was later extended by Sundaram \cite{Sunderam90} 
to arbitrary shapes, providing  a combinatorial proof of the Cauchy identity for the symplectic group.
 Following this idea, Chen et.al.~\cite{CDDSY07}  introduced the (simplified) vacillating tableaux 
and gave a bijection between  set partitions and vacillating tableaux of empty shape, which
 demonstrated the symmetry between the maximal crossings and the maximal nestings in set partitions. 
This idea was further extended to fillings of Ferrers shapes, stack polyominoes, and moon polyominoes, under 
the language of growth diagrams,  to describe the properties of monotone subsequences in those fillings. 
For more information, see \cite{GP2020,Kratten06, Rubey11}. 

Given their versatility in algebra and combinatorics, it is crucial to understand the properties of vacillating tableaux and their relationships with other mathematical objects.
The objective of this short paper is to study the vacillating tableaux defined by the map $DI_n^k$ for integer sequences, and it is more convenient to use the simplified notion to describe our result. 
For a fixed integer sequence $\bsy{i}\in [n]^k$,  one observes that 
$P^k_{m}(\bsy{i})$ can be different from
$P^k_{m+1} (\bsy{i})$, even if one considers the corresponding simplified version. However, we  show 
in Section~\ref{sec:2} that 
when $m$ is sufficiently large, the simplified vacillating tableaux  $(P^k_{m}(\bsy{i}))^*$
become  stable, which is denoted by $P^*(\bsy{i})$ and 
is called the limiting vacillating tableau 
of $\bsy{i}$.  We will give a characterization of the set of limiting vacillating tableaux.   In Section~\ref{sec:3},
we enumerate the number of limiting vacillating tableaux of length $2k$. First, we express this number as a sum of Stirling numbers of the second kind and the number of involutions in 
the symmetric group, and then we prove that this coincides with the sequence \href{https://oeis.org/A004211}{A004211} in OEIS \cite{OEIS}, which counts certain bi-colored set partitions.

\section{Limiting vacillating tableaux}\label{sec:2}

First, we recall the bijection $DI_n^k$ of \cite{HL05} to give the
exact construction of the $n$-vacillating tableau $P^k_{n}(\bsy{i})$ for an integer sequence $\bsy{i}$. 
 In the following  an integer partition is visually represented by the Young diagram, which contains $\lambda_j$ boxes in the $j$-th row.  We use English notation in which the diagrams are aligned in the upper-left corner.
 A \emph{partial tableau} is an integer partition whose boxes  are filled  with distinct positive integers that are increasing in rows and columns. 
The set of integers in the boxes of a partial tableau $T$, denoted content$(T)$,  is call the \emph{content} of~$T$. 
A \emph{standard Young tableau} with $n$ boxes is a partial tableau whose content is exactly~$[n]$.

The  main ingredients of the bijection $DI_n^k$ are the RSK row insertion algorithm and a special case of  Sch\"{u}tzenberger's jeu de taquin algorithm, which removes a box containing an entry $x$ from a partial tableau and produces a new partial tableau. We adopt the description from \cite{HL05}, while the full version and in-depth discussion of these algorithms can be found in \cite[Chapter 3]{Sagan01} or \cite[Chapter 7]{EC2}.

\medskip 

\noindent 
\underline{The RSK row insertion}. Let $T$ be a partial tableau of partition shape $\lambda$ with $|\lambda| <n$ and with district entries from $\{1,2,\dots, n\}$. Let $x$ be a positive integer that is not in $T$. The operation $x \xrightarrow{RSK}  T$ is defined as follows. 
\begin{enumerate}[(a)]
    \item Let $R$ be the first row of $T$. 
    \item While $x$ is less than some element in $R$, do 
       \begin{enumerate}[label=\roman*)]
           \item Let $y$ be the smallest element of $R$ greater than $x$; 
           \item Replace $y \in R$ with $x$;
           \item Let $x:=y$ and let $R$ be the next row. 
       \end{enumerate}
    \item Place $x$ at the end of $R$ (which is possibly empty). 
\end{enumerate}
The result is a partial tableau of shape $\mu$ such that $|\mu /\lambda|=1$. For each occurrence of step (b), we say that $x$ \emph{bumps} $y$ to the next row.  

\medskip 

\noindent 
\underline{Jeu de taquin}. \ 
Let $T$ be a partial tableau of partition shape $\lambda$  with distinct entries in $\{1,2,\dots, n\}$. Let $x$ be an entry in $T$.  The following operation will delete $x$ from $T$ and yield a partial tableau $S$.
\begin{enumerate}[(a)]
    \item Let $c=T_{i,j}$  be the box of $T$ containing $x$, 
    which is  the $j$-th box in the $i$-th row of $T$. 
    \item While $c$ is not a corner, i.e., a box both at the end of a row and the end of a column, do
      \begin{enumerate}[label=\roman*)]
          \item Let $c'$ be the box containing $\min\{T_{i+1,j}, T_{i, j+1}\}$; if only one of $T_{i+1,j}, T_{i, j+1}$ exists, then the minimum is taken to be that single entry; 
          \item Exchange the positions of $c$ and $c'$. 
      \end{enumerate}
    \item Delete $c$. 
\end{enumerate}
We denote this process by $x \xleftarrow{jdt} T$. 

It is well-known that both of the above two processes are invertible. Examples  can be found in  \cite[Section 3]{HL05}. 
The bijection $DI_n^k$ from integer sequences to the set of pairs 
consisting of a SYT and an $n$-vacillating tableau is built by iterating alternatively between the two processes.

\medskip 

\noindent 
\underline{The bijection $DI_n^k$.} \ 
Let $(i_1, i_2, \dots, i_k) \in [n]^k$ be an integer sequence of length $k$.  First we define a sequence of partial tableaux recursively: the $0$-th
tableau is the unique SYT of shape $(n)$ with the filling $1, 2, \dots, n$, namely, 
\begin{center} 
\begin{tikzpicture}
\node[left] at (0,0) {$T^{(0)} =$}; 
\draw (0,-.25) rectangle (3,.25); 
\draw (0.5, -.25)--(0.5,0.25) (1,-0.25)--(1,0.25) (2.5, -0.25)--(2.5, 0.25); 
\node at (0.25,0) {$1$}; 
\node at (0.75, 0) {$2$}; 
\node at (1.75,0) {$\dots$}; 
\node at (2.75,0) {$n$}; 
\node at (3.2, -0.2) {.}; 
\end{tikzpicture}
\end{center} 

Then for integers $j=0, 1, \dots, k-1$, the partial tableaux $T^{(j+\frac{1}{2})}$ and $T^{(j+1)}$ are defined by
\begin{eqnarray}
T^{(j+\frac{1}{2})} &= &\left( i_{j+1} \xleftarrow{jdt} T^{(j)} \right), \\ T^{(j+1)}  & = & \left( i_{j+1} \xrightarrow{RSK} T^{(j+\frac{1}{2})} \right). 
\end{eqnarray}
Note that $T^{(j+1)}$ is always a SYT. 
Let $\lambda^{(i)}$ be the shape of $T^{(i)}$ for all indices $i$,   and $\lambda=\lambda^{(k)}$ be the shape of the last partial tableau $T^{(k)}$. Finally, let 
\begin{eqnarray}
T^k_{n}(\bsy{i})=T^{(k)}, \qquad P^k_{n}(\bsy{i})
=\left(  
\lambda^{(0)}, \lambda^{(\frac{1}{2})}, \lambda^{(1)}, \lambda^{(1\frac{1}{2})}, \dots, \lambda^{(k)} 
\right),  
\end{eqnarray}
so that $T^k_{n}(\bsy{i})$ is a SYT of shape $\lambda$ and 
$P^k_{n}(\bsy{i})$ is an $n$-vacillating tableau of shape $\lambda$ and length $2k$.  The image of $\bsy{i}$ under the map $DI_n^k$ is given by  $DI_n^k(\bsy{i}) =  \left(T^k_{n}(\bsy{i}), 
P^k_{n}(\bsy{i})\right)$. 

\ytableausetup{smalltableaux, aligntableaux=center,baseline}
\begin{example} \label{DI_k's}
 Consider the integer sequence $\bsy{i}=(4,4)$.  First let $n=4$ and apply the map $DI_4^2$. The $(T^{(j)})$ sequence is 
 \[\begin{ytableau}
1 &2 &3 &4
\end{ytableau} \ \to\  \begin{ytableau}
1 &2 &3
\end{ytableau}\ \to\  \begin{ytableau}
1 &2 &3 &4
\end{ytableau}\ \to\  \begin{ytableau}
1 &2 &3
\end{ytableau}\ \to\ \begin{ytableau}
1 &2 &3 &4
\end{ytableau}\,.\] Hence for $DI_4^2(\bsy{i})$,  $T^2_{4}(\bsy{i})=$\ytableausetup{smalltableaux}
\begin{ytableau}
1 &2 &3 &4
\end{ytableau}
, and \ytableausetup{smalltableaux}
\[P^2_{4}(\boldsymbol{i}):\left (\ydiagram{4}\,,\  \ydiagram{3}\,,\ \ydiagram{4}\,,\ \ydiagram{3}\,,\   \ydiagram{4} \right).
\] It follows that the simplified vacillating tableau for $P^2_{4}(\bsy{i})$ is $(\emptyset,\emptyset,\emptyset,\emptyset,\emptyset)$.  

For $n = 5$, we apply the map $DI^2_5$ to obtain the $(T^{(j)})$ sequence as follows.
 \[\begin{ytableau}
1 &2 &3 &4 &5
\end{ytableau} \ \to\ \begin{ytableau}
1 &2 &3 &5
\end{ytableau}\ \to\  \begin{ytableau}
1 &2 &3 &4\\
5
\end{ytableau}\ \to\  \begin{ytableau}
1 &2 &3\\
5
\end{ytableau}\ \to\ \begin{ytableau}
1 &2 &3 &4\\
5
\end{ytableau}.\] Similarly, for $n = 6$, we apply the map $DI^2_6$ to obtain the following $(T^{(j)})$ sequence:
\[\begin{ytableau}
1 &2 &3 &4 &5 & 6
\end{ytableau} \ \to\  \begin{ytableau}
1 &2 &3 &5 &6
\end{ytableau}\ \to\  \begin{ytableau}
1 &2 &3 &4 &6\\
5
\end{ytableau}\ \to\  \begin{ytableau}
1 &2 &3 &6\\
5
\end{ytableau}\ \to\ \begin{ytableau}
1 &2 &3 &4\\
5 &6
\end{ytableau}\,.\]
Thus for $DI_5^2(\bsy{i})$,  $T^2_{5}(\boldsymbol{i})= \ytableausetup{smalltableaux} \begin{ytableau}
1 &2 &3 &4\\
5
\end{ytableau}$, and \[ P^2_{5}(\boldsymbol{i})=  \left(\ydiagram{5},\, \ydiagram{4}, \ \ydiagram{4,1},\ \ydiagram{3,1},\ \ydiagram{4,1} \right);\] 
for $DI_6^2(\bsy{i})$,  $T^2_{6}(\boldsymbol{i})= \ytableausetup{smalltableaux}\begin{ytableau}
1 &2 &3 &4\\
5 &6
\end{ytableau}$, and \[ P^2_{6}(\boldsymbol{i})=  \left(\ydiagram{6},\, \ydiagram{5}, \ \ydiagram{5,1},\ \ydiagram{4,1},\ \ydiagram{4,2} \right).\]
     
      
     

 The simplified vacillating tableau for $ \ytableausetup{smalltableaux}
P^2_{5}(\bsy{i})$ 
 is $(\emptyset, \emptyset,  \ydiagram{1}\,,\ydiagram{1}\,, \ydiagram{1})$, while the 
 simplified vacillating tableau  for 
  $P^2_{6}(\bsy{i})$ is 
  $\left(\emptyset, \emptyset, \ydiagram{1}\,,\ydiagram{1}\,, \ydiagram{2}\right)$. 
 \end{example} 

From the above example we see that the vacillating tableaux defined by $DI_n^k(\bsy{i})$ and
$DI_{m}^k(\bsy{i})$ may not be the same, even if we consider the simplified version. 
Nevertheless, we prove that for any  integer sequence $\bsy{i}$, 
the simplified vacillating tableaux 
$(P^k_{m}(\bsy{i}))^*$ eventually stabilize when $m$ is sufficiently large.

\begin{theorem} \label{thm: exist} 
Let $\bsy{i}=(i_1, \dots, i_k) \in [n]^k$ be an integer sequence. For $m > n+2k$, the vacillating tableaux $P^k_{m}(\bsy{i})$
 and $P^k_{m+1}(\bsy{i})$ have the same simplified version, i.e.,~the two vacillating tableaux become the same after removing the first row from each integer partition in the sequences.
\end{theorem}
\begin{proof} 
Consider the sequences of partial tableaux $(T_m^{(j)})$ and $(T_{m+1}^{(j)})$
determined by the maps $DI_m^k(\bsy{i})$ and $DI_{m+1}^k(\bsy{i})$, respectively. We will prove by induction that $T_{m+1}^{(j)}$ can be obtained from $T_m^{(j)}$ by adding the entry $m+1$ at the end of the first row,
for all $j$ such that $0 \leq j \leq k$ and $j \in \mathbb{Z}/2$.
The claim is obviously true for $j=0$. 
{In the following, we let $j$ represent integers only and prove that if the claim is true for $j$, then it is true for the indices $j+\frac{1}{2}$ and $j+1$. } 

For any partial tableau involved, there are at most $k$ boxes below the first row. 
Thus the length of the first row is at least $m-k$, which is larger than $n+k$ and is strictly larger than the length of the second row.
When an entry $i_j \leq n$  is inserted 
in the application of either $DI_m^k$ or $DI_{m+1}^k$,  it 
can  bump at most one integer in $[n+1, n+k]$  to the second row. Hence  any integer larger than  $n+k$ 
 stays in the first row in all the partial tableaux  during the delete-insert process. 

For integers $j=0, 1, \dots, k-1$, 
\begin{enumerate}
    \item Consider the tableau $T_*^{(j+\frac{1}{2})}$ obtained by $(i_{j+1} \xleftarrow{jdt} T_*^{(j)})$, where $*$ is $m$ or $m+1$. 
    
    \begin{enumerate} 
    \item If the shapes of $T_m^{(j +\frac{1}{2})}$ and $T_m^{(j)}$ differ by a corner box in the first row, then 
    the jeu de taquin process only shifts boxes in the first row of $T_m^{(j)}$. Since the first row of 
    $T_m^{(j)}$ is strictly longer than its second row, and $T_{m+1}^{(j)}$ has one more entry, $m+1$, at the end of the first row, applying jeu de taquin to $T_{m+1}^{(j)}$  will only involve the boxes in the first row and have $m+1$ at the end,  and hence leave the simplified tableau unaffected.

    \item If the shape of $T_m^{(j +\frac{1}{2})}$ and $T_m^{(j)}$ differ by a corner box not in the first row, then the boxes exchanged in $(i_{j+1} \xleftarrow{jdt} T_{m+1}^{(j)})$ 
    are exactly the same as those boxes exchanged in $(i_{j+1} \xleftarrow{jdt} T_{m}^{(j)})$, which do not involve the box containing  $m$. Hence, the entry $m+1$ in $T_{m+1}^{(j)}$ stays at  the end of the first row and has never been involved  in the deleting process. 
    \end{enumerate} 
    
    \item Consider the tableau  $T_*^{(j+1)}$ obtained by $(i_{j+1} \xrightarrow{RSK} T_*^{(j+\frac{1}{2})})$, where $*$  is $m$ or $m+1$. 
    
    Since the entry $m$ is in the same position of the first row of $T_m^{(j+\frac{1}{2})}$ and $T_{m+1}^{(j+\frac{1}{2})}$, and $i_{j+1} \leq n < m$, we know $i_{j+1}$ will bump an entry no larger than $ n+j+1$ to the second row in the RSK insertion. In other words, all the entries larger than $n+j+1$ 
    in $T_{m+1}^{(j+1)}$ 
    will stay in the first row, including $m+1$.   
\end{enumerate}

 Now removing the first rows from $(T_m^{(j)})$ 
 and $(T_{m+1}^{(j)})$ for all $j$ will give us the same integer partition sequence; that is,   $P^k_{m}(\bsy{i})$
 and $P^k_{m+1}(\bsy{i})$ have the same simplified version. 
\end{proof} 

For any integer sequence $\bsy{i} \in [n]^k$, let $P^*(\bsy{i})$ be the simplified vacillating tableau $\left(P^k_{m}(\bsy{i})\right)^*$ when $m > n+2k$. By Theorem \ref{thm: exist}, $P^*(\bsy{i})$ is well-defined.  We call $P^*(\bsy{i})$ the limiting vacillating tableau  of the integer sequence $\bsy{i}$. It is in the set 
$\mathcal{SVT}_k(\mu)$ for some 
 ending shape $\mu$ with at most
$k$ boxes.

Next result gives a characterization of the 
limiting vacillating tableaux for integer sequences of length $k$.   In Section 3, we give explicit formula for   the number of such vacillating tableaux. 
Note that two integer sequences may have  the same 
limiting vacillating tableau, for example, any sequence 
$(a, b)$ with $a > b$ has the limiting vacillating tableau 
$(\emptyset, \emptyset, \square,  \square, \ytableausetup{aligntableaux=center,boxsize=7px}\ydiagram{1,1} )$. Nevertheless, not all simplified vacillating tableaux 
can be the limit for some integer sequences. We prove that 
a vacillating tableau is a limiting one if and only if 
a box is added whenever the index gets from $j+\frac{1}{2}$ 
to $j+1$, for all integer $j$. For instance, for the simplified vacillating tableaux in Example \ref{DI_k's}, the one for $P_{6}^2(\bsy{i})$ is a limiting one, while the ones for $P_{4}^2(\bsy{i})$
and $P_{5}^2(\bsy{i})$ are not.

\begin{proposition} \label{prop: limit}
A sequence of integer partitions $(\lambda^{(j)}: j=0, \frac{1}{2}, 1, 1\frac{1}{2}, \dots, k)$ is the  limiting 
vacillating tableau for some integer sequence of length $k$ if and only if $\lambda^{(0)}=\emptyset$ and the sequence satisfies the following
conditions for each integer $j=0, 1, \dots, k-1$: 
\begin{enumerate}[(a)]
\item $\lambda^{(j)} \supseteq  \lambda^{(j+\frac12)}$ and 
    $| \lambda^{(j)} / \lambda^{(j+\frac12)}|=0$ or $1$, 
    \item $\lambda^{(j+\frac12)}\subseteq \lambda^{(j+1)}$ and 
    $|\lambda^{(j+1)}/ \lambda^{(j+\frac12)}|=1$. 
    
    \label{b-part}
\end{enumerate}
\end{proposition}
\begin{proof} 
From the proof of Theorem~\ref{thm: exist}, 
in the process of applying $DI_m^k$, 
when we use jeu de taquin to delete an entry,  we remove one box either from the first row, in which case the simplified tableau has the same shape,  or from a row below the first, resulting in a simplified tableau with  a box removed. When we perform RSK insertion, since $m$ is sufficiently large, the entry $m$ always stays in the first row,  which guarantees that an entry is bumped  down, that is, a box must be added in the simplified tableau. 
In conclusion,  when $m$ is sufficiently large,  the simplified vacillating tableau satisfies the conditions (a) and (b). 

Conversely, let $P=(\lambda^{(j)}: j=0, \frac{1}{2}, 1, 1\frac{1}{2}, \dots, k)$   be 
a simplified vacillating tableau of length $2k$ that satisfies the conditions listed. Assume there are $\ell$ boxes in $\lambda^{(k)}$, then we have  $\ell \leq k$. Take an integer $n$ such that  $n >2k$. 
Let $\tilde{P} = (\tilde{\lambda}^{(j)})$  be the sequence of integer partitions obtained from $P$ by  adding a row of boxes to each $\lambda^{(j)}$ such that $\tilde{\lambda}^{(j)}$ has $n$ boxes if $j$ is an integer, and $n-1$ boxes if $j$ is a half-integer.  Let $F$ be a SYT of shape $\tilde{\lambda}^{(k)}$ that has $1, 2, \dots, 2k-\ell, 2k+1, 2k+2, \dots, n$ in the first row.  
Since the map $DI_n^k$ is a bijection, 
 there is a unique integer sequence $\bsy{i}=(i_1, \dots, i_k) \in [n]^k$ such that $DI_n^k(\bsy{i})=(F, \tilde{{P}})$.  
 By Condition \ref{b-part}, each time we perform the RSK algorithm to insert $i_j$, 
 an entry is bumped down from the first row. 
 Hence $i_j \neq n$ for all $j$. Furthermore,  $n$ cannot be bumped into the second row during the entire 
 delete-insert process, otherwise it will stay under the first row and be an entry in $F$, 
 which contradicts the construction of $F$. 
 Therefore the entry $n$ must stay at the end of the first row when we apply $DI_n^k(\bsy{i})$. 
 
Now let $m > n$ and compare $DI_n^k(\bsy{i})$ and $DI_m^k(\bsy{i})$.   When we construct the partial tableaux $(T^{(j)})$ in $DI_m^k(\bsy{i})$, 
the integers $n+1, \dots, m$  always stay at the end of the first row and do not interfere with the delete-insert process.  In other words, the simplified vacillating tableau obtained from $DI_m^k(\bsy{i})$ is the same as the one obtained from $DI_n^k(\bsy{i})$, which is ${P}$.
This proves that ${P}$ is a limiting vacillating tableau.  
\end{proof} 

The diagram in Figure~\ref{Diagram} gives the limiting vacillating tableaux of length up to 3.  The number next to each integer partition indicates the number of simplified vacillating tableaux ending at that  integer partition.

\begin{figure}[ht] 
\begin{center} 
    \begin{tikzpicture}
     \node at (0,0) {$k=0:$};  
     \node at (2,0) {$\emptyset$}; 
     \node[red] at (2.3,-0.2) {\small $1$}; 
     \draw[->] (2, -0.3)--(2, -1.1); 
     
     \node at (0, -1.5) {$k=\frac{1}{2}:$}; 
     \node at (2, -1.5) {$\emptyset$}; 
     \node[red] at (2.3, -1.7) {\small{$1$}}; 
     \draw[->] (2, -1.8)--(2, -2.6); 
     
     \node at (0, -3) {$k=1$:}; 
     \draw (1.8, -3.2) rectangle (2.2, -2.8); 
     \node[red] at (2.5, -3.2) {\small $1$}; 
     \draw[->] (2,-3.3)--(2,-4.1); \draw[->](2.2, -3.3)--(3.2, -4.1); 
     
     \node at (0, -4.5) {$k=1\frac{1}{2}:$}; 
     \node at (2, -4.5) {$\emptyset$};
     \node[red] at (2.3, -4.7) {\small $1$}; 
     \draw (3.0, -4.7) rectangle (3.4, -4.3); 
     \node[red] at (3.7, -4.7) {\small$1$}; 
     \draw[->] (2, -4.8)--(2, -5.6); 
     \draw[->] (3.2, -4.8)--(3.2, -5.6);
     \draw[->] (3.4, -4.8)--(4.8, -5.6);

     \node at (0, -6) {$k=2:$}; 
     \draw (1.8, -6.2) rectangle (2.2, -5.8); 
     \node[red] at (2.5, -6.2) {\small$1$}; 
     \draw (3.0, -6.2) rectangle (3.8, -5.8); 
     \draw (3.4, -5.8)--(3.4, -6.2); 
     \node[red] at (4.1, -6.2) {\small$1$}; 
     \draw (4.6, -5.8) rectangle (5.0, -6.6); 
     \draw (4.6, -6.2)--(5.0, -6.2); 
     \node[red] at (5.3, -6.6) {\small$1$}; 
     \draw[->] (2, -6.5)--(2, -7.3); 
     \draw[->] (2.2, -6.5)--(3.2, -7.3); 
     \draw[->] (3.4, -6.5)--(3.4, -7.3); 
     \draw[->] (3.6, -6.5)--(5.2, -7.3); 
     \draw[->] (5.0, -6.7)--(7.2, -7.4); 
     \draw[->] (4.8, -6.7)--(3.6, -7.3);

     \node at (0, -7.9) {$k=2\frac{1}{2}:$}; 
     \node at (2, -7.9) {$\emptyset$}; 
     \node[red] at (2.3, -8.2) {\small$1$}; 
     \draw (3.2, -8.1) rectangle (3.6, -7.7); 
     \node[red] at  (3.9, -8.2) {\small$3$}; 
     \draw (4.8, -8.1) rectangle (5.6, -7.7); 
     \draw (5.2, -8.1)--(5.2, -7.7); 
      \node[red] at (5.9, -8.2) {\small$1$}; 
     \draw (7, -7.6) rectangle (7.4, -8.4); 
       \draw (7, -8)--(7.4, -8); 
       \node[red] at (7.7, -8.2) {\small$1$}; 
      \draw[->] (2, -8.3)--(2, -9.1);  
      \draw[->] (3.4, -8.3)--(3.4,-9.1); 
      \draw[->] (3.6,-8.3)--(4.6, -9.1); 
      \draw[->] (5.2, -8.3)--(5.8, -9.1); 
      \draw[->] (5.4, -8.3)--(7.6, -9.1); 
      \draw[->] (7.2, -8.5)--(7.8, -9.1); 
      \draw[->] (7.4, -8.5)--(9.6, -9.1);

     \node at (0, -9.6) {$k=3:$}; 
     \draw (1.8, -9.8) rectangle (2.2, -9.4); 
      \node[red] at (2.5, -9.8) {\small$1$}; 
     \draw (3.0, -9.8) rectangle (3.8, -9.4); 
      \draw (3.4, -9.8)--(3.4, -9.4); 
      \node[red] at (4.1, -9.8) {\small$3$}; 
     \draw (4.6, -9.4) rectangle (5.0, -10.2); 
       \draw (4.6,-9.8)--(5.0, -9.8); 
       \node[red] at (5.3, -10.2) {\small$3$}; 
     \draw (5.8, -9.8) rectangle (7, -9.4); 
        \draw (6.2,-9.8)--(6.2,-9.4) (6.6,-9.8)--(6.6,-9.4); 
        \node[red] at (7.3, -9.8) {\small$1$}; 
    \draw (7.8, -9.8) rectangle (8.6,-9.4); 
    \draw (7.8,-9.4) rectangle (8.2,-10.2); 
    \node[red] at (8.6, -10.2) {\small$2$}; 
    \draw (9.4, -9.4) rectangle (9.8, -10.6); 
      \draw (9.4, -9.8) rectangle (9.8, -10.2); 
      \node[red] at (10.1, -10.6) {\small$1$}; 
    \end{tikzpicture}
    \caption{Limiting vacillating tableaux of length up to $3$.}  \label{Diagram} 
    \end{center} 
    \end{figure}
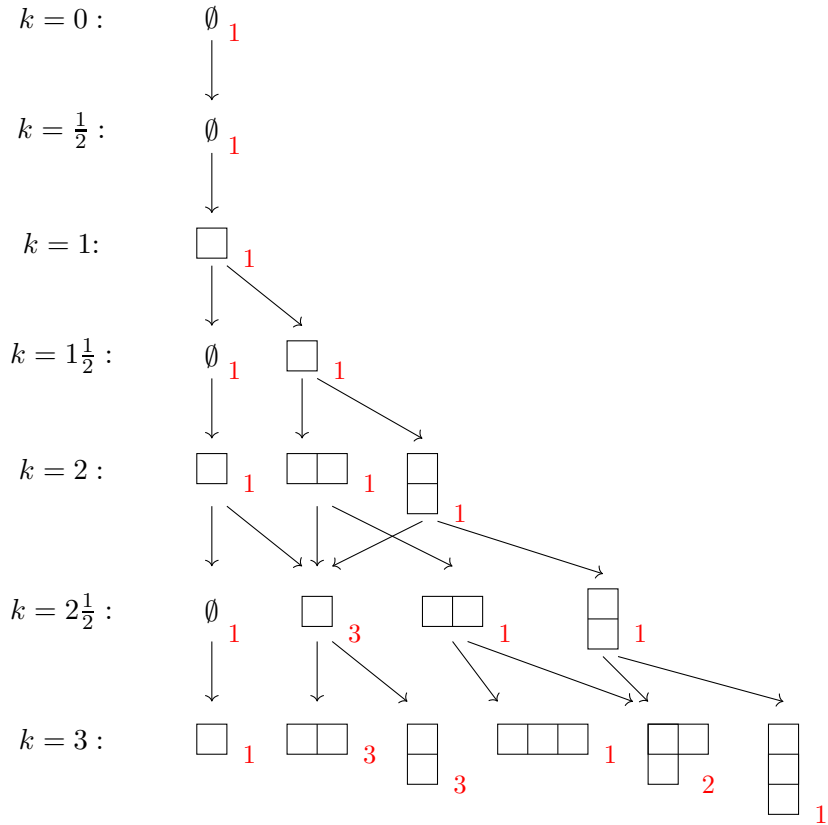

\section{Enumeration of the limiting vacillating tableaux }\label{sec:3}
Let $a_k$ be the number of limiting vacillating tableaux of length $2k$. We set $a_0=1$. The initial values of the infinite sequence $(a_k)_{k=0}^\infty$ are $1, 1, 3, 11, 49,257, \dots.$ Checking the data on OEIS \cite{OEIS} reveals that these values coincide with the initial values of the sequence \href{https://oeis.org/A004211}{A004211}, which is described as ``shifts one place left under 2nd-order binomial transformation" and ``equals the eigensequence of the box of Pascal's triangle.'' 
A combinatorial interpretation  of the sequence A004211 is
given by Goyt and Pudwell \cite{GP11}, who proved that the $k$-th term, denoted $c_k$, counts the number of arrangements of the multiset $\{ 1, 1, 2, 2, \dots, k,k\}$ satisfying the following two conditions: 
\begin{enumerate}[(a)]
    \item All entries between the two occurrences of any given value $i$ exceed $i$, and 
    \item No three entries increase from left to right with the last two adjacent.
\end{enumerate}
Furthermore, $c_k$  can be computed by the formula 
\begin{equation} \label{A004211} 
c_k=\sum_{m=1}^k 2^{k-m} \begin{Bmatrix}
k\\m
\end{Bmatrix}, 
\end{equation} 
where $\genfrac\{\}{0pt}{2}{k}{m}$ is the Stirling number of the second kind that counts the number of set partitions of $[k]$ into $m$ blocks. 
Let $\mathcal{P}_k$ be the set of set partitions of the set $[k]$ such that each element is colored either red or blue, and for each block the minimal element is colored red. We call such a set partition a \emph{bi-colored set partition of $[k]$}. 
Clearly the summation in \eqref{A004211} counts the number of bi-colored set partitions in $\mathcal{P}_k$. Goyt and Pudwell proved \eqref{A004211} by constructing a bijection between the set $\mathcal{P}_k$ and the set of the special arrangements of the multiset $\{ 1, 1, 2, 2, \dots, k, k\}$ described above.

We claim that the sequence $(a_k)_{k=0}^\infty$ is exactly A004211. 
To prove this, first we give an explicit formula for $a_k$ in terms of 
Stirling numbers of the second kind and numbers of standard Young tableaux. Then we show this formula is equivalent to \eqref{A004211} through a bijective proof. 

\begin{proposition}  \label{summation2} 
Let $k$ be a positive integer. Then 
\begin{equation} \label{eq: limit2} 
a_k = \sum\limits_{j=1}^k \Bigg(\begin{Bmatrix}
k\\j
\end{Bmatrix}\sum\limits_{\lambda \vdash j}f^\lambda\Bigg).
\end{equation} 
\end{proposition}

\begin{proof}
We will show that there is a bijection between the set of limiting vacillating tableaux  of length   $2k$  and the set of all pairs
$$
\{(\bsy{B},T): \bsy{B} \mbox{ is a set partition of  $[k]$ and   $T$  is a partial tableau with } \textrm{content}(T)=\max(\bsy{B}) \},
$$ 
where $\max(\bsy{B})$ is the set of maximal elements in the blocks of $\bsy{B}$. 
The bijection is based on the  one constructed in \cite[Section 2]{CDDSY07} to count the total number of simplified vacillating tableaux. Here we restrict the bijection to the set of limiting vacillating tableaux as described in Proposition~\ref{prop: limit}. We remark that the basic operation of this bijection  was first given by
Sundaram \cite{Sunderam90} in the study of the Cauchy identity in the symplectic group $Sp(2n)$. 

\medskip 

\noindent 
\textbf{The Bijection $\psi$ from  limiting vacillating tableaux to pairs $(\bsy{B},T)$}.  \ \ 
Given a simplified vacillating tableau  ${P} = (\emptyset=\lambda^{(0)}, \lambda^{(\frac{1}{2})}, \lambda^{(1)},\dots, \lambda^{(k)})$ satisfying the two conditions listed in Proposition~\ref{prop: limit}, we will recursively define a sequence $(E_0, T_0), (E_{\frac{1}{2}}, T_{\frac{1}{2}}), 
\dots, (E_{k}, T_{k})$ where $E_i$  is a set of ordered pairs of integers in $[k]$ (which are viewed as ``edges"), and $T_j$ is a partial tableau of shape $\lambda^{(j)}$.  Let $E_0$  be the empty set and  $T_0$ be the empty tableau (on the empty alphabet). For each integer $j=1,2, \dots, k$, 
\begin{enumerate}[(a)] 
   \item If $\lambda^{(j-\frac{1}{2})} = \lambda^{(j-1)}$, then $(E_{j-\frac{1}{2}}, T_{j -\frac{1}{2}}) = (E_{j-1}, T_{j-1})$.  
   \item  If $\lambda^{(j-\frac{1}{2})} \subsetneq  \lambda^{(j-1)}$,  let $T_{j-\frac{1}{2}}$ be the unique partial tableau (on a suitable alphabet) of shape $\lambda^{(j-\frac{1}{2})}$  such that $T_{j-1}$ is obtained from $T_{j-\frac{1}{2}}$  by row-inserting some integer $m$. Note that $m$ must be less
than $j$. Let $E_{j-\frac{1}{2}}$ be obtained from $E_{j-1}$ by adding the ordered pair $(m, j)$.
  \item  For a limiting vacillating tableau we always have 
   $\lambda^{(j)} \supsetneq  \lambda^{(j-\frac{1}{2})}$. 
   Let $E_j =E_{j-\frac{1}{2}}$. 
and $T_j$ is obtained from $T_{j-\frac{1}{2}}$ by adding the entry $j$ in the box  
$ \lambda^{(j)}  / \lambda^{(j-\frac{1}{2})}$. 

\end{enumerate} 
It is clear from the above construction that $E_0 \subseteq E_{\frac{1}{2}} \subseteq \cdots \subseteq  E_{k}$. 
Let $G=(V, E_k)$ be a graph with the vertex set $V=[k]$ and edge set $E_k$, and $\bsy{B}$ be the set partition of $[k]$ whose blocks are connected components of $G$. 
Finally we define $\psi({P}) = (\bsy{B}, T_k)$.

In \cite{CDDSY07} it is shown that  $\psi$ is a bijection from the set of all simplified vacillating tableaux of length $2k$ to the pairs of $(\bsy{B}, T)$ where $\bsy{B}$ is a set partition of $[k]$ and $T$ is a partial tableau  with content$(T) \subseteq \max(\bsy{B})$. 
The equality may not hold for a general simplified vacillating tableau, since it allows $\lambda^{(j)} = \lambda^{(j-\frac{1}{2})}$
for an integer $j$. (In that case just let $E_j=E_{j-\frac{1}{2}}$ and $T_j = T_{j-\frac{1}{2}}$.)  
 If ${P}$ is a limiting vacillating tableau,  then each integer $i$ either appears as the first component of an ordered pair in $E_k$,  or appears  in 
the last tableau $T_k$.  Hence in this case $T=T_k$ is a partial tableau  with 
content$(T)=\max(\bsy{B})$. 

Now we count the number of pairs $(\bsy{B}, T)$. 
{  For each integer $j \in [k]$, there  are $\sum\limits_{\lambda \vdash j} f^{\lambda}$ choices for STYs with $j$ boxes  and $\displaystyle\genfrac\{\}{0pt}{2}{k}{j}$ set 
 partitions of $[k]$ with $j$ blocks.}  Replacing the content set $[j]$ with the set of maximal block elements in the set partition yield the pair $(\bsy{B}, T)$. 
 Hence \eqref{eq: limit2}  follows. 
\end{proof}

\begin{example}
An example of the map $\psi$ follows. Let the limiting
vacillating tableau ${P}$ be the following sequence with $k=4$: 
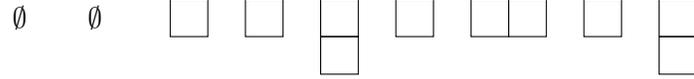
\begin{center} 
\begin{tikzpicture}
\node at (0,0) {$\emptyset$}; 
\node at (1,0) {$\emptyset$}; 
\draw (2,-0.25) rectangle (2.5, 0.25);
\draw (3, -0.25) rectangle (3.5, 0.25);
\draw (4, -0.25) rectangle (4.5, 0.25); \draw (4,-0.25) rectangle (4.5,-0.75); 
\draw (5, -0.25) rectangle (5.5, 0.25);
\draw (6, -0.25) rectangle (6.5, 0.25); \draw (6.5, -0.25) rectangle (7, 0.25);
\draw (7.5, -0.25) rectangle (8, 0.25);
\draw (8.5, -0.25) rectangle (9, 0.25); \draw (8.5, -0.25) rectangle (9, -0.75);
\end{tikzpicture}
\end{center} 

Then the sequence of $(E_i, T_i)$ is 
\begin{center} 
\begin{tikzpicture}
\node at (-2, 1) {$i$:}; 
\node at (0,1) {$0$}; 
\node at (1,1) {$\frac{1}{2}$}; 
\node at (2.25,1) {$1$}; 
\node at (3.25, 1) {$1\frac{1}{2}$}; 
\node at (4.25, 1) {$2$}; 
\node at (5.25, 1) {$2\frac{1}{2}$}; 
\node at (7, 1) {$3$}; 
\node at (8.75, 1) {$3\frac{1}{2}$}; 
\node at (10.25, 1) {$4$};
\node at (-2,0) {$T_i$:}; 
\node at (0,0) {$\emptyset$}; 
\node at (1,0) {$\emptyset$}; 
\draw (2,-0.25) rectangle (2.5, 0.25); \node at (2.25,0) {1}; 
\draw (3, -0.25) rectangle (3.5, 0.25); \node at (3.25,0) {1};
\draw (4, -0.25) rectangle (4.5, 0.25); \draw (4,-0.25) rectangle (4.5,-0.75); 
   \node at (4.25,0) {1}; \node at (4.25, -0.5) {2}; 
\draw (5, -0.25) rectangle (5.5, 0.25); \node at (5.25,0) {2}; 
\draw (6.5, -0.25) rectangle (7, 0.25); \draw (7, -0.25) rectangle (7.5, 0.25);
   \node at (6.75, 0) {2}; \node at (7.25,0) {3}; 
\draw (8.5, -0.25) rectangle (9, 0.25);  \node at (8.75,0) {2}; 
\draw (10, -0.25) rectangle (10.5, 0.25); \draw (10, -0.25) rectangle (10.5, -0.75);
   \node at (10.25,0) {2}; \node at (10.25, -0.5) {4}; 
   
 \node at (-2, -1.5) {$E_i$:};   
   \node at (0, -1.5) {$\emptyset$}; 
   \node at (1, -1.5) {$\emptyset$}; 
   \node at (2.25, -1.5) {$\emptyset$}; 
   \node at (3.25, -1.5) {$\emptyset$}; 
   \node at (4.25, -1.5) {$\emptyset$}; 
   \node at (5.25, -1.5) {$\{(1,3)\}$}; 
   \node at (7, -1.5) {$\{(1,3)\}$}; 
    \node at (8.75, -1.5) {$\{(1,3)$}; \node at (8.75,-2) {$(3,4)\}$}; 
   \node at (10.25, -1.5) {$\{(1,3)$}; \node at (10.25,-2) {$(3,4)\}$}; 
\end{tikzpicture}
\end{center} 
and  hence $\psi(P)=(\bsy{B}, T_4)$, where $\bsy{B}$ is the set  partition $\{\{134\}, \{2\}\}$.
\end{example}

\begin{corollary}
Let $k$ be a positive integer. Then 
\begin{eqnarray} \label{cor8} 
a_k = \sum\limits_{j=1}^k \Bigg(\begin{Bmatrix}
k\\j
\end{Bmatrix}(\# \mbox{ involutions in $\mathfrak{S}_j$})\Bigg).
\end{eqnarray} 
It follows that  the exponential generating function of the sequence $(a_k)_{k=0}^\infty$ is
\begin{eqnarray} \label{egf} 
A(x)=\sum_{n \geq 0} a_k  \frac{x^k}{k!} = \exp\left( \frac{e^{2x}-1}{2}\right). 
\end{eqnarray} 
\end{corollary}
\begin{proof}
Equation~\eqref{cor8}  follows from the fact that the number of involutions in the symmetric group $\mathfrak{S}_j$ equals the number of  SYT with $j$ boxes, which is given by  $\sum\limits_{\lambda \vdash j}f^\lambda$.

By \eqref{cor8} $a_k$ also counts the number of ways to partition a set of $k$ elements into disjoint non-empty blocks, then place an involution on the blocks. By the Compositional Formula of exponential generating functions, $A(x)= G(F(x))$, where $F(x)$ and $G(x)$ are the exponential generating functions for the Bell numbers and the involutions, respectively. See \cite[Theorem 5.1.4]{EC2}.
Explicitly, 
\[
F(x) = \exp(e^x-1), \qquad \qquad G(x) = \exp\left( x + \frac{x^2}{2}\right). 
\]
Putting them together, we obtain the forluma of $A(x)$. 

\end{proof}

 The generating function $A(x)$
matches that of the sequence A004211 given in OEIS, implying that $(a_k)_{k=0}^\infty$ is the sequence A004211.  Next we present a bijective proof using Formulas \eqref{A004211}  and \eqref{cor8}.  

\begin{theorem} \label{2-counting} 
Let $k$ be a positive integer. Then 
\begin{equation}  \label{partitions} 
\sum\limits_{j=1}^k \Bigg(\begin{Bmatrix}
k\\j
\end{Bmatrix}(\# \mbox{ involutions in $\mathfrak{S}_j$})\Bigg)\Bigg) = \sum\limits_{t=1}^k \Bigg(\begin{Bmatrix}
k\\t
\end{Bmatrix}2^{k-t}\Bigg).
\end{equation}
Consequently, the sequence $(a_k)_{k=0}^\infty$ is exactly the OEIS sequence \href{https://oeis.org/A004211}{A004211}. 
\end{theorem}

\begin{proof}
For a set partition of $[k]$ with $j$ blocks, we 
order the blocks according to the minimal element in each block. That is, 
if $\bsy{B} = \{B_1, \dots, B_j\}$ is a set partition of $[k]$, then $\min{B_1} < \min{ B_2} < \cdots < \min {B_j}$. Let $\sigma \in \mathfrak{S}_j$ be an involution. Then   $\sigma$ is a product of disjoint 1 or 2-cycles.
Denote by $\mathcal{B}_k$ the set of all pairs 
$\{( \bsy{B}, \sigma)\}$ where $\bsy{B}$ is a set partition of $[k]$ and $\sigma$ is an involution whose length equals the number of blocks of $\bsy{B}$. 

Let $\mathcal{P}_k$ be the set of bi-colored set partitions of $[k]$ such that each number in $[k]$ is colored either red or blue, and the minimal element of each block is colored red.  
By \cite{GP11}, $\mathcal{P}_k$ 
is enumerated by the OEIS sequence \href{https://oeis.org/A004211}{A004211} and 
$|\mathcal{P}_k|$ can be expressed by the right side of \eqref{partitions}. 

We define  a bijection $\phi: \mathcal{B}_k
\mapsto \mathcal{P}_k$  as follows. Let 
$(\bsy{B}, \sigma)  \in \mathcal{B}_k$ where $B=\{B_1, \dots, B_j\}$. 

\begin{enumerate}[(a)]
    \item Construct a set partition.  
    For every 2-cycle $(a,b)$ in $\sigma$, we form a new block by merging $B_a$ with $B_b$. For every 1-cycle $(r)$ in $\sigma$, we keep $B_r$ as a new block.  The collection of all the new blocks forms $\bsy{B'}$. 
    
    \item In each block $B$ of $\bsy{B}'$,  color the  element $\min{B}$  red. For each non-minimal element $x$ of $B$,  color it red if and only if 
    $x$ and $\min{B}$ are in the same block of $\bsy{B}$. Otherwise, $x$ is blue. 
     \end{enumerate}
This gives a bi-colored set partition in the set $\mathcal{P}_k$. 

Conversely, for a bi-colored set partition 
$\bsy{B'}$ in $\mathcal{P}_k$, where the minimal element in each block is red, we can recover the pair $(\bsy{B}, \sigma)$ by the following steps. 
\begin{enumerate}[(a')]
    \item First for each block of $\bsy{B'}$, split it into two blocks by putting the red elements into one block and blue elements into another block. If there is no blue elements, then we just keep it as one block. 
    
    The collection of all the blocks obtained this way forms  $\bsy{B}$. 
    
    \item Assume there are $j$ blocks in $\bsy{B}$. 
    Label the blocks of $\bsy{B}$ as $B_1, B_2, \dots, B_j$ according to their minimal elements so that
     $\min{B_1} < \min{B_2} < \cdots < \min{B_j}$.  For any block $B_i$,  if there is another block $B_k$ such that $B_i \cup B_k$ is one   block of $\bsy{B'}$, then $(i, k)$ is a 2-cycle of $\sigma$. Otherwise, $(i)$ is a 1-cycle of $\sigma$. This creates an involution $\sigma$ in $\mathfrak{S}_j$. 
\end{enumerate}

Note that this process gives $\phi^{-1}$ which proves that $\phi$ is bijective and \eqref{partitions} follows.
\end{proof}

\begin{example} 
We give an example illustrating the bijection 
$\phi$ in the proof of Theorem~\ref{2-counting}. Let $k=10$,  $\bsy{B}=\{ \{ 1, 3, 4, 8\}, \{2, 5\}, \{ 6\}, \{ 7, 9, 10\}\}$, and 
$\sigma=(14)(23)$ in cycle notation.  Then the set 
partition $\bsy{B'}$ is obtained from $\bsy{B}$ by merging $B_1$ with $B_4$ and merging $B_2$ with $B_3$. The coloring is indicated by the superscript: 
$\bsy{B'}=\{ \{ 1^r, 3^r, 4^r, 7^b, 8^r, 9^b, 10^b\}, \{ 2^r, 5^r, 6^b\}\}$. 
\end{example}

\section{Final Remarks and Future Projects} 

The limiting vacillating tableaux induce an equivalence relation on the integer sequences. 
So it is natural to consider the following problem. 

\medskip 

\noindent \emph{Problem 1. When do two integer sequences have the same limiting vacillating tableau?}  

\medskip 
For $k=2$ we have a complete answer, 
which can be checked easily.    

\begin{example}
Let $P^*(\bsy{i})$ be the limiting vacillating tableau of the integer sequence $\bsy{i}= (i_1, i_2)$, where both $i_1$ and $i_2$ are positive integer. Then
\begin{enumerate}[(a)]
    \item $P^*(\bsy{i})=( \emptyset, \emptyset, \square, \emptyset, \square)$ if and only if $i_2=i_1+1$ or $i_1=i_2=1$. 
    \item $P^*(\bsy{i})=( \emptyset, \emptyset, \square,  \square, \ytableausetup{boxsize=7px}\ydiagram{2})$ if and only if $i_1=i_2 >1$,  or $i_1 < i_2-1$. 
    \item $P^*(\bsy{i})=( \emptyset, \emptyset, \square,  \square, \ytableausetup{aligntableaux=center,boxsize=7px}\ydiagram{1,1})$ if and only if $i_1 > i_2$.
 \end{enumerate}
\end{example}

The answer for longer sequences is much more complicated and the general case remains open.   The  proofs of Proposition \ref{prop: limit}, \ref{summation2}, 
and Theorem \ref{2-counting} suggest some possible approaches and related questions, which we plan to investigate next. 

In the proof of Proposition \ref{prop: limit}, for a simplifed vacillating tableau $P$ satisfying the conditions listed,  we constructed an integer 
sequence $\bsy{i}$ such that $P^*(\bsy{i})=P$, using the bijective map $DI_n^k$ with $n > 2k$. In the proof we used a SYT $F$, which is clearly not unique. Hence a related problem is: 

\medskip 

\noindent \emph{Problem 2. Fix $n$ and $k$ with $n \geq 2k$. Determine which integer sequences correspond to the same vacillating tableau under the bijection $DI_n^k$. That is, find the criterion such that for two integer sequences
$\bsy{i}$ and $\bsy{j}$, $DI_{n}^k(\bsy{i})$ and  $DI_n^k(\bsy{j})$ yield the same shape $\lambda$ and 
$P_{n}^k(\bsy{i}) = P_{n}^k(\bsy{j})$. 
} 

\medskip 
In Section~\ref{sec:3}, we find bijections from the set of limiting vacillating tableaux of length $2k$ 
to certain families of enriched set partitions, for example,  the set of set partitions each with an involution built on its blocks, or the set  $\mathcal{P}_k$ of bi-colored set partitions.  
This suggests the following problem. 

\medskip 

\noindent \emph{Problem 3. What can one say about the integer sequences based on some given information on the structures of the corresponding enriched set partitions?  For example, does the number of blocks, or the number of fixed points in the involution, have an interpretation in the integer sequences? 
} 

\section*{Acknowledgement}
 We thank ICERM for hosting  the ``Research Community in Algebraic Combinatorics" program in 2021, which facilitated and supported this research project.  The second author is supported by an EDGE Karen Uhlenbeck fellowship, and the fourth author is supported in part by the Simons Collaboration Grant for Mathematics 704276.


\begin{thebibliography}{10}

\bibitem{BH17} G.~Benkart, T.~Halverson. 
Partition Algebras and the Invariant Theory of the Symmetric Group. 
in Recent Trend in Combinatorics, 
ed:  H\'{e}l\`{e}ne Barcelo, Gizem Karaali, Rosa Orellana. 
Association for Women in Mathematics Series, volume 16. 2019. 

\bibitem{BHH17} 
G.~Benkart, T.~Halverson, and N.~Harman, Dimensions of irreducible modules for partition algebras and tensor power multiplicities for symmetric and alternating groups, 
\emph{J.
Algebr. Combin}. 46 no. 1 (2017), 77-108.

\bibitem{Berele86} A.~Berele, A Schensted-type correspondence for the symplectic group, 
\emph{J. Combin. Theory Ser. A} 43 (1986). 320-328.

\bibitem{CDDSY07} 
W.Y.C.~Chen, E.Y.P.~Deng, R.R.X.~Du, R.P.~Stanley and C.H.~Yan.
{Crossings and nestings of matchings and set partitions }.  
\emph{Trans. Amer. Math. Soc}. volume 359, no. 4, (2007) 1555--1575.


\bibitem{COSSZ} L.~Colmenarejo, R.~Orellana,  F.~Saliola, A.~Schilling, and M.~Zabrocki. An insertion algorithm of multiset partitions with applications to diagram algebras. 
\emph{Journal of Algebra}, 
557, 1(2020),  97-128.



\bibitem{GP11} 
A.M.~Goyt and L.K.~Pudwell. Solution to Monthly  problem 11567. May, 2011. 
 \url{https://faculty.valpo.edu/lpudwell/papers/MonthlyStirling.pdf}
 

\bibitem{GP2020} T.~Guo and S.~Poznanovi\'{c}. Hecke insertion and maximal increasing and de- creasing sequences in fillings of stack polyominoes. 
\emph{J. Combin. Theory A}, 176(2020), 105304.


\bibitem{HJ18} T.~Halverson and T.N.~Jacobson. 
Set-partition tableaux and representations of diagram algebras. 
\emph{Algebraic Combinatorics}, Volume 3 (2020) no. 2, 509-538.


\bibitem{HL05} T.~Halverson and  T.~Lewandowski. 
RSK insertion for set partitions and diagram algebras.
\emph{Electron J. Combin.} 11 (2004/2006), 24. 


\bibitem{HT10} T.~Halverson and N.~Thiem. 
$q$-Partition algebra combinatorics. 
\emph{J. Combin. Theory A}, 117(2010), 507--527. 

\bibitem{Kratten06} C.~Krattenthaler. Growth diagrams, and increasing and decreasing chains in fillings of Ferrers shapes. 
\emph{Adv. Appl. Math.}, 37 (2006), 404–431.


\bibitem{Rubey11} M.~Rubey. Increasing and decreasing sequences in fillings of moon polyominoes. 
\emph{Adv. Appl. Math.}, 47(2011), 57–87.

\bibitem{OEIS} N.J.A.~Sloane. The On-line Encyclopedia of Integer Sequences. \url{https://oeis.org. } 

\bibitem{Sagan01} B.~Sagan. the Symmetric Groups, Representations, Combinatorial Algorithms, and Symmetric Functions, second edition, Springer-Verlag, New York, 2001. 


\bibitem{EC2} R.P.~Stanley. Enumerative Combinatorics, Vol. 2, Cambridge University Press, Cambridge, 1999. 

\bibitem{Sunderam90} 
S.~Sundaram, The Cauchy identity for Sp(2n), \emph{J. Combin. Theory Ser. A}, 53 (1990), 209–238.
\end{thebibliography}
\end{document}